\documentclass{birkjour}
\usepackage{amsthm, amssymb, amsmath, amsrefs}

\newtheorem{theorem}{Theorem}[section]

\newtheorem{prop}[theorem]{Proposition}

\theoremstyle{definition}
\newtheorem{definition}[theorem]{Definition}
\newtheorem{remark*}[theorem]{Remark}
\newtheorem{example}[theorem]{Example}

\newcommand{\D}{\mathbb{D}}
\newcommand{\C}{\mathbb{C}}
\newcommand{\T}{\mathbb{T}}

\begin{document}

%
%
%
%
%
%
%
%
%

\title{Hadamard multipliers of the Agler class}
\author{Greg Knese}

\address{Department of Mathematics, Washington University in St. Louis, St. Louis, MO 63130, USA}
\email{geknese@wustl.edu}




\date{\today}
\dedicatory{In memory of Rien Kaashoek}



\begin{abstract}
We prove several results about functions which preserve
the Schur-Agler class under Hadamard or coefficient-wise product.
First, functions which preserve the Schur class necessarily
preserve the Schur-Agler class.
Second, ``moments'' of certain commuting operator tuples form
coefficients of Schur-Agler class preservers.
Finally, any preserver of the full matrix Schur-Agler class must
have coefficients given by moments of commuting operator tuples.
We also point out that any counterexample to the multivariable
von Neumann inequality can be used to derive a non-trivial
Agler class preserver.

\end{abstract}

\maketitle

\section{Introduction}
Given two power series $f(z) = \sum_{\alpha\in \mathbb{N}_0^d} f_{\alpha} z^{\alpha}$, 
$F(z) = \sum_{\alpha \in \mathbb{N}_{0}^{d}} F_{\alpha} z^{\alpha}$
we define the Hadamard or coefficient-wise product
\[
(f \star F)(z) = \sum_{\alpha\in \mathbb{N}_0^d} f_{\alpha} F_{\alpha} z^{\alpha}.
\]
Given a space of analytic functions $X$ 
we often wish to characterize those analytic functions $F$
with the property that for every $f \in X$, we have $f\star F \in X$.
When this happens we will say that $F$ preserves $X$
under Hadamard product.  The question is only
interesting for spaces that are not easily described
in terms of properties of their coefficients.
A sampling of results for various function spaces 
are in the papers \cites{Caveny, MR, RS, Tov0, Tov}. 

We define the \emph{Schur class} $\mathcal{S}_d$ to be 
the set of analytic functions on the unit polydisk $\D^d \subset \C^d$ that are
bounded by $1$ in supremum norm.  
Sheil-Small characterized the functions that 
preserve the Schur class.  
Even more, his result allows for restrictions
of the class based on power series support.
Given a subset $L \subset \mathbb{N}_0^d$ we write 
$\mathcal{S}_d(L)$ 
for the elements of $\mathcal{S}_d$ whose power series coefficients are
supported on the set $L$; namely, functions
with power series 
of the form $f(z) = \sum_{\alpha \in L} f_{\alpha} z^{\alpha}$.

\begin{theorem}[Sheil-Small \cite{sheilsmall} Theorem 1.3]
Let $L \subset \mathbb{N}_0^d$.
The function $F(z) = \sum_{\alpha \in \mathbb{N}_0^d} F_{\alpha} z^{\alpha}$ preserves $\mathcal{S}_d(L)$ under Hadamard product
(in the sense that for all $f \in \mathcal{S}_d(L)$ we have $f\star F \in \mathcal{S}_d(L)$)
if and only if there is a complex measure 
$\mu$ on the $d$-torus $\T^d$ of total variation $\|\mu\|\leq 1$ such
that for every $\alpha \in L$
\[
F_{\alpha} = \int_{\T^d} w^{\alpha} d\mu(w).
\]
In particular, for $L = \mathbb{N}_0^d$ we have
\[
F(z) = \int_{\T^d} \frac{1}{\prod_{j=1}^{d} (1-z_j w_j)} d\mu(w).
\]
\end{theorem}

Sheil-Small used the Hadamard product to 
study coefficient problems for univalent functions.
Thus, many of the main applications of \cite{sheilsmall}
are subsumed by de Branges's establishment of the
Bieberbach conjecture \cite{deBranges}.

Some basic examples of preservers of the Schur class are
\[
F_{\zeta}(z) = \frac{1}{\prod_{j=1}^{d} (1-z_j \zeta_j)}
\]
for $\zeta \in \overline{\D}^d$.
Fej\'er means (or Cesaro sums) of Schur class functions
can be viewed as a Hadamard product and it is well-known
such functions preserve the Schur class.
A perhaps more interesting class of examples of $F$ preserving $\mathcal{S}_d$
is given by taking $b \in \mathcal{S}_d$ with $b(0)=0$
and $F(z) = \frac{1}{1-b(z)}$. This is (2.6) of \cite{sheilsmall}.
This includes the ``diagonal extraction'' multiplier $\frac{1}{1-z_1\cdots z_d}$.

In this paper we investigate the coefficient
multipliers of an important class of analytic
functions called the Schur-Agler class---we shall write
Agler class for short.

\begin{definition}[\textbf{Agler class}]
\label{Agler}
Let $d\in \mathbb{N}$ and let \(f: \D^d \rightarrow \overline\D\) be an analytic function. 
Then we say \(f\) is in the \emph{Agler class} if for every \(d\)-tuple \(T=(T_1,...,T_d)\) of commuting, strictly contractive operators on a Hilbert space, we have \(\|f(T)\| \leq 1\). 
We use \(\mathcal{A}_d\) to denote the \emph{Agler class}.
\end{definition}

Since the operators in $T$ above are strictly contractive, $f(T)$ can be 
defined as an absolutely convergent power series.  
Membership in the Agler class can be tested using $d$-tuples of commuting
strictly contractive \emph{matrices} instead of operators on infinite dimensional
Hilbert spaces.  
This is proven as part of a more general fact in
\cite{polyhedraAMY} (Theorem 6.1)---in particular one only needs to
check over the class of simultaneously diagonalizable $d$-tuples of commuting
strictly contractive matrices where the joint eigenspaces have dimension $1$ 
(these are called ``generic'').

We shall also use $\mathcal{A}_d(L)$ to denote those
elements of $\mathcal{A}_d$ whose power series are supported on $L \subset \mathbb{N}_0^d$.
Von Neumann's inequality \cite{vN} and And\^{o}'s inequality \cite{ando}
prove that $\mathcal{A}_1=\mathcal{S}_1$ and $\mathcal{A}_2=\mathcal{S}_2$.
On the other hand, work of Varopoulos \cite{varo}
shows that $\mathcal{A}_d \ne \mathcal{S}_d$ for $d>2$.
The Agler class is of interest both as a vehicle to study why and how the multivariable
von Neumann inequality fails but also because Agler class functions exhibit many
interesting properties.  In particular, Agler class functions have a special
decomposition into positive semi-definite kernel functions called an Agler
decomposition (that certifies membership in the Agler class)
as well as a characterizing formula called a transfer function realization;
see the book \cite{AMbook}.
The current paper uses neither of these facts and instead relies
on connections to the theory of completely bounded maps.
Some interesting recent papers
about the Agler class are \cite{Barik},  \cite{Bhowmik}, \cite{Debnath}.

A natural problem is to characterize the $F$ 
that preserve the Agler class $\mathcal{A}_d(L)$.
A first step is to notice that Schur class
preservers are Agler class preservers.

\begin{theorem} \label{SAthm}
Every $F$ which preserves $\mathcal{S}_d(L)$ under Hadamard product
necessarily preserves $\mathcal{A}_d(L)$.
\end{theorem}
We will give a direct proof of this fact even though some later results
also prove this.  While this result is not deep, it proves some
non-obvious facts like a Bernstein inequality for $\mathcal{A}_d$.
Recall the Bernstein inequality says that for a one variable polynomial $P(z) \in \C[z]$ of degree $n$,
\[
\sup_{|z|=1}|P'(z)| \leq n \sup_{|z|=1} |P(z)|.
\]
(The same is true for Laurent polynomials of degree $n$.)
This can be viewed as saying $F(z) = \sum_{j=1}^{\infty} (j/n) z^j = \frac{z}{n(1-z)^2}$
is a coefficient multiplier of $\mathcal{S}_1(\{0, \dots, n\})$.

For a multivariable polynomial $P(z_1,\dots, z_d) \in \C[z_1,\dots, z_d]$
of multidegree $n = (n_1,\dots, n_d)$
this trivially extends in the form
\[
\sup_{\T^d} \left|\frac{\partial P}{\partial z_j} \right| \leq n_j \sup_{\T^d} |P|.
\]
By the above result applied with 
\[
L = \{ (m_1,\dots, m_d) \in \mathbb{N}_0^d: 0\leq m_j \leq n_j \text{ for } 1\leq j \leq d\}
\]
we see that if $P \in \mathcal{A}_d$ then 
\[
\frac{1}{n_j} \frac{\partial P}{\partial z_j} \in \mathcal{A}_d.
\]

It is also worth pointing out that elements of $\mathcal{S}_d$ themselves 
preserve $\mathcal{S}_d$ under $\star$ product and hence preserve the
Agler class $\mathcal{A}_d$ as well.

On the other hand, the functions that 
preserve $\mathcal{A}_d$ under $\star$ product are definitely different.

\begin{prop} \label{propexists}
For $d>2$, there exists $F$ which preserves $\mathcal{A}_d$ under Hadamard product
but does not preserve $\mathcal{S}_d$.  
\end{prop}

The existence comes from counterexamples to von Neumann inequalities
in 3 or more variables.
For instance, the Holbrook example (a modification of the Kaijser-Varopoulos examples) \cite{Holbrook}
yields $F(z_1,z_2,z_3) = z_1^2+z_2^2+z_3^2 - (1/2)(z_1z_2+z_2z_3+z_1z_3)$
which preserves the Agler class but not the Schur class.
The Crabb-Davie example \cite{CrabbDavie} yields 
$F(z_1,z_2,z_3) = z_1z_2z_3 - z_1^3-z_2^3-z_3^3$
which preserves the Agler class but not the Schur class.
In this case, the polynomial from the Crabb-Davie counterexample happens to
be the same as the Hadamard multiplier function.
Just exhibiting these examples proves some operator inequalities.
For example if $a_{ij}$ are scalars such that
\[
\| a_{11} T_1^2 + a_{22} T_2^2 + a_{33} T_3^2 + a_{12} T_1T_2 + a_{23} T_2 T_3 + a_{13} T_1 T_3\| \leq 1
\]
holds for all $3$-tuples of commuting contractions $T=(T_1,T_2,T_3)$,
then 
\[
\| a_{11} T_1^2 + a_{22} T_2^2 + a_{33} T_3^2 -(1/2)(a_{12} T_1T_2 + a_{23} T_2 T_3 + a_{13} T_1 T_3)\| \leq 1
\]
for all such 3-tuples.

Our method of proving $F$ has the desired properties is from
the following observation. 

\begin{theorem}\label{basicthm}
Let $L\subset \mathbb{N}_0^d$.
For any $d$-tuple $T=(T_1,\dots, T_d)$ of commuting
contractive operators on a Hilbert space $\mathcal{H}$ and for vectors $x,y \in \mathcal{H}$ with
$|x|,|y|\leq 1$, 
the function
\[
F(z) = \langle \prod_{j=1}^{d} (1-z_j T_j)^{-1} x,y \rangle 
= \sum_{\alpha \in \mathbb{N}_0^d} \langle T^{\alpha} x,y\rangle z^{\alpha}
\]
preserves $\mathcal{A}_d(L)$ under Hadamard product.  
\end{theorem}

If we let $\mu$ be a complex Borel measure on $\T^d$ with total variation $\|\mu\|\leq 1$
then multiplication $M_{j}$ by $z_j$ on $L^2(|\mu|)$ is a contractive operator.
So the $d$-tuple $M = (M_1,\dots, M_d)$ is a $d$-tuple of commuting contractions and
we have $\langle M^{\alpha} \frac{d\mu}{d|\mu|}, 1\rangle = \int_{\T^d} z^{\alpha} d\mu$
which recovers all of the coefficients from Sheil-Small's theorem.
This separately proves that a Schur class preserver is an Agler class preserver.

The functions $F$ in Theorem \ref{basicthm} have a stronger property;
they actually preserve matrix-valued Agler classes.
If $f:\D^d \to \C^{M \times N}$ is analytic and $f(z) = \sum_{\alpha} f_{\alpha} z^{\alpha}$
we say $f$ is in the matrix Agler class if for every commuting $d$-tuple $T$
of strictly contractive operators we have
\[
\left\| \sum_{\alpha\in \mathbb{N}_0^d} T^{\alpha} \otimes f_{\alpha}  \right\| \leq 1
\]
where $T^{\alpha} \otimes f_{\alpha}$ is the tensor
product of operators (on Hilbert spaces).
Once we notice this, we are led to a characterization.
We let $\mathcal{MA}_d$ denote the matrix Agler class,
and $\mathcal{MA}_d(L)$ denotes those elements with 
power series supported on the set $L\subseteq \mathbb{N}_0^d$.

\begin{theorem} \label{matrixthm}
Let $L\subseteq \mathbb{N}_0^d$.
Given $F(z) = \sum_{\alpha} F_{\alpha} z^{\alpha}$ 
analytic on $\D^d$, $F$ preserves $\mathcal{MA}_d(L)$ under Hadamard product 
if and only if
there exists a $d$-tuple $T=(T_1,\dots, T_d)$ of commuting
contractive operators on a Hilbert space $\mathcal{H}$ and vectors $x,y \in \mathcal{H}$
with $|x|,|y|\leq 1$
such that for $\alpha \in L$
\[
F_{\alpha} = \langle T^{\alpha} x,y\rangle.
\]
\end{theorem}

The proof is a straighforward application of the Stinespring representation
theorem and the Wittstock extension theorem as presented in 
Paulsen's book \cite{Paulsenbook}.

A simple corollary is now possible.
If $F(z) = \sum_{\alpha} F_{\alpha} z^{\alpha}$ 
preserves $\mathcal{MA}_{d}$, then for any $\beta \in \mathbb{N}_0^d$
so does $G(z)  = \sum_{\alpha} F_{\alpha+\beta} z^{\alpha}$.
Indeed, representing $F_{\alpha} = \langle T^{\alpha} x,y\rangle$ as in the theorem
we have $F_{\alpha+\beta} = \langle T^{\alpha} T^{\beta}x,y\rangle$.
This is an analogue of (2.1) from \cite{sheilsmall}.

The remainder of the paper consists of proofs of the theorems listed above.

\section{Proofs}

\begin{proof}[Proof of Theorem \ref{SAthm}]
Suppose $F$ preserves $\mathcal{S}_d(L)$ under the Hadamard product.
Let $f \in \mathcal{A}_d(L)$.  We need to show $f \star F \in \mathcal{A}_d(L)$.
Now, given a $d$-tuple $T$ of contractive commuting operators 
on a Hilbert space $\mathcal{H}$ and vectors $x,y \in \mathcal{H}$ of norm at most $1$,
we set $z \odot T = (z_1T_1,\dots, z_d T_d)$ and have that
\[
h(z) := \langle f(z\odot T) x,y \rangle  = \sum_{\alpha \in L} f_{\alpha} \langle T^\alpha x,y\rangle z^{\alpha}
\]
belongs to the Schur class.
Therefore, $h \star F$ belongs to the Schur class.
However, 
\[
(h\star F)(z) = \sum_{\alpha \in L} f_{\alpha} F_{\alpha} \langle T^\alpha x,y\rangle z^{\alpha}
 = \langle (f\star F)(z\odot T)x,y\rangle.
\]
Since this holds for all unit vectors $x,y$, we see $(f\star F)(z\odot T)$ is contractive
and since $T$ is an arbitrary $d$-tuple of commuting contractive operators 
we see that $f\star F \in \mathcal{A}_d(L)$.
\end{proof}

\begin{proof}[Proof of Proposition \ref{propexists}]
For concreteness we refer to the minimal counter-example of the von Neumann
inequality due to Holbrook \cite{Holbrook} (it is a modification of the
example of Kaijser-Varopoulos--- appendix of \cite{varo}).
The operator tuple $T=(T_1,T_2,T_3)$ on $\C^4$ is
built by taking orthonormal vectors $e,f$ and an orthogonal
subspace $G= \text{span}\{g_1,g_2,g_3\}$ where 
$g_1, g_2, g_3$ are unit vectors satisfying $\langle g_i, g_j \rangle = -1/2$
for $i\ne j$.  We can take vectors congruent to the three vectors 
\[
(1,0), (-1/2, \sqrt{3}/2), (-1/2, -\sqrt{3}/2).
\]
Now, for $j=1,2,3$,  $T_j e = g_j$, $T_j f = 0$, 
\[
T_j g_i = \begin{cases} f & \text{ if } j=i \\
\frac{-1}{2}f & \text{ if } j\ne i. \end{cases}
\]
The tuple $T=(T_1,T_2,T_3)$ is well-defined, pairwise commuting, contractive, and satisfies the
relations
\[
\langle T_i e,f\rangle = 0 \text{ for } j =1,2,3
\]
\[
\langle T_i^2 e, f \rangle = 1 \text{ for } j =1,2,3
\]
\[
\langle T_i T_j e, f \rangle = -1/2 \text{ for } i\ne j
\]
and
\[
\langle T_{\alpha} e, f \rangle = 0 \text{ for } |\alpha| >2.
\]
Thus, 
\[
F(z) = \sum_{\alpha} \langle T^{\alpha} e,f\rangle z^\alpha = 
z_1^2+z_2^2+z_3^2 -(1/2)(z_1z_2 + z_1z_3+ z_2z_3).
\]

Holbrook proves that for
$P(z_1,z_2,z_3) = z_1^2+z_2^2+z_3^2 - 2(z_1z_2+z_2z_3+z_1z_3)$ we have
\[
\|P(T)\| = \langle P(T) e,f\rangle = 6 > \sup_{z\in \T^d} |P(z)|.
\]
If $L_P =\text{supp}(P) = \{(2,0,0),(0,2,0),(0,0,2),(1,1,0),(1,0,1),(0,1,1)\}$
then 
\[
(P \star F)(1, 1, 1) = \langle P(T) e,f\rangle
\]
shows that $P/\|P\|_{\infty} \in \mathcal{S}_3(L_P)$ but $(P/\|P\|_{\infty}) \star F \notin \mathcal{S}_3$.
Theorem \ref{basicthm}, discussed momentarily, proves that $F$ preserves $\mathcal{A}_3(L)$ for every $L$.
Thus, we have produced an Agler class preserver that does not preserve the Schur class.
\end{proof}

\begin{example}
Similarly, using the counterexample from \cite{CrabbDavie}, 
we obtain a triple $T=(T_1,T_2,T_3)$ of commuting contractive matrices operating 
on $\C^8$ with orthogonal unit vectors $x,y \in \C^8$ (if referring to the paper \cite{CrabbDavie} we take $x=e$ and $y =h$)
such that
\[
\langle T_i x,y \rangle  = 0, \text{ for } i=1,2,3
\]
\[
 \langle T_iT_j x,y \rangle  = 0,  \text { for } i,j = 1,2,3
 \]
 \[
\langle T_j T_i^2 x,y \rangle = -\delta_{ij}, \text{ for } i,j= 1,2,3
\]
\[
\langle T_1T_2T_3 x,y \rangle = 1, \text{ and } \langle T^{\alpha} x,y \rangle = 0 \text{ for } |\alpha| >3.
\]
So we obtain the Agler class Hadamard multiplier
\[
F(z_1,z_2,z_3) = z_1z_2z_3 - z_1^3-z_2^3-z_3^3. \quad \diamond
\]
\end{example}

Theorem \ref{basicthm} is a special case of the `if' direction of Theorem \ref{matrixthm}.
To prove this we make use of the standard tensor product of Hilbert spaces
and the tensor product of bounded operators on Hilbert space.
The main fact we need is that the tensor product of contractive operators is 
contractive.

\begin{proof}[Proof of `if' direction of Theorem \ref{matrixthm}]
To be precise suppose we have
a $d$-tuple $T=(T_1,\dots, T_d)$ of commuting
contractive operators on a Hilbert space $\mathcal{H}$ and vectors $x,y \in \mathcal{H}$
with $|x|,|y|\leq 1$
where we define 
\[
F(z) = \langle \prod_{j=1}^{d} (1-z_j T_j)^{-1} x,y \rangle = \sum_{\alpha\in \mathbb{N}_0^d} \langle T^{\alpha} x,y\rangle z^{\alpha}.
\]
Let $L \subset \mathbb{N}_0^d$ and suppose
$f:\D^d \to \C^{M\times N}$ is analytic, $f(z) = \sum_{\alpha\in L} f_{\alpha} z^{\alpha}$ 
(here each $f_\alpha \in \C^{M\times N}$), and belongs to the matrix Agler class.

We must show 
\[
(f \star F)(z) = \sum_{\alpha \in L} f_{\alpha} \langle T^{\alpha} x,y\rangle z^{\alpha}
\]
belongs to the matrix Agler class. The above power series is absolutely
convergent for each $z \in \D^d$ since $|\langle T^{\alpha} x,y\rangle| \leq 1$.
So, let $S=(S_1,\dots, S_d)$ be a $d$-tuple of commuting strict contractions
on a Hilbert space $\mathcal{K}$ and consider
\[
(f\star F)(S) =   \sum_{\alpha \in L} \langle T^{\alpha} x,y\rangle S^{\alpha}\otimes f_{\alpha} 
\]
which converges absolutely.  
For any unit vectors $v\in \mathcal{K} \otimes \C^{N}$, $w \in \mathcal{K}\otimes \C^{M}$, 
\[
\langle (f\star F)(S)v,w\rangle = 
 \sum_{\alpha \in L} \langle T^{\alpha} x,y\rangle \langle (S^{\alpha} \otimes f_{\alpha})v,w\rangle
\]
but
\[
\langle T^{\alpha} x,y\rangle \langle (S^{\alpha} \otimes f_{\alpha})v,w\rangle
 = \langle ((T\otimes S)^{\alpha} \otimes f_{\alpha} ) (x\otimes v), (y\otimes w)\rangle 
\]
where $T\otimes S = (T_1\otimes S_1,\dots, T_d\otimes S_d)$ 
is a commuting $d$-tuple of strict contractions on $\mathcal{K}\otimes \mathcal{H}$
because $\|T_j\otimes S_j\| \leq \|T_j\| \|S_j\| <1$ for each $j$.
We also view $x\otimes v$ as an element of $\mathcal{H}\otimes \mathcal{K}\otimes \C^N$
and $y\otimes w$ as an element of $\mathcal{H}\otimes \mathcal{K}\otimes \C^M$.
So,
\[
\langle (f\star F)(S)v,w\rangle 
= \langle (\sum_{\alpha \in L} (T\otimes S)^{\alpha} \otimes f_{\alpha})(x\otimes v), (y\otimes w)\rangle
=
\langle (f(T\otimes S)(x\otimes v), (y\otimes w)\rangle.
\]
Thus, since $f$ belongs to the matrix Agler class, $f(T\otimes S)$ is a contraction
so that $(f\star F)(S)$ is a contraction.  This proves $F$ preserves
the matrix Agler classes.
\end{proof}

To prove the other direction of Theorem \ref{matrixthm} we use
the Wittstock extension theorem \cite{Wittstock84} as stated in \cite{Paulsenbook} (Theorem 8.2)
as well as the completely bounded version of the Stinespring's representation theorem 
(\cite{Paulsenbook} Theorem 8.4).

\begin{theorem}[Wittstock's extension theorem] \label{wittstock}
Let $\mathcal{A}$ be a unital $C^*$-algebra, $\mathcal{M}$ a subspace of $\mathcal{A}$,
and let $\phi: \mathcal{M} \to B(\mathcal{H})$ be completely bounded.
Then, there exists a completely bounded map $\psi: \mathcal{A} \to B(\mathcal{H})$
that extends $\phi$, with $\|\phi\|_{cb}=\|\psi\|_{cb}$.
\end{theorem}

\begin{theorem}[Stinespring representation for completely bounded maps] \label{stinespring}
Let $\mathcal{A}$ be a $C^*$-algebra with unit, and let $\phi:\mathcal{A}\to B(\mathcal{H})$
be a completely bounded map. Then there exists a Hilbert space $\mathcal{K}$,
a $*$-homomorphism $\pi: \mathcal{A} \to B(\mathcal{K})$, and bounded operators
$V_i: \mathcal{H}\to \mathcal{K}$, $i=1,2$, with $\|\phi\|_{cb} = \|V_1\|\|V_2\|$
such that
\[
\phi(a) = V_1^* \pi(a) V_2
\]
for all $a\in \mathcal{A}$. Moreover, if $\|\phi\|_{cb}=1$, then $V_1,V_2$ may
be taken to be isometries.
\end{theorem}

We refer the reader to \cite{Paulsenbook} for definitions of all of the concepts
above.  

In order to apply the above theorems, we need to 
build an appropriate $C^*$-algebra.  The following
is a standard construction. 
All that needs to be
done is build an operator $d$-tuple $T$ (of commuting contractive
operators) that is
universal for all matrix $d$-tuples in the sense
that 
for any matrix valued polynomial $P \in \C^{M_1\times M_2}[z_1,\dots, z_d]$
and any $d$-tuple $S$ of commuting contractive matrices we have 
\[
\|P(T)\| \geq \|P(S)\|.
\]
As before, $P(T)$ is defined as $\sum_{\alpha} T^\alpha \otimes P_{\alpha}$
for $P(z) = \sum_{\alpha} P_\alpha z^{\alpha}$.
Once we have $T=(T_1,\dots, T_d)$ built then our $C^*$-algebra is simply
the $C^*$-algebra $\mathcal{A}$ generated by $T_1,\dots, T_d$.
For each size $N$ we can find a countable dense subset of
the collection of $d$-tuples of commuting contractive $N\times N$ matrices.
Taking a countable direct sum of these $d$-tuples varying over all $N = 1, 2,\dots$
we can obtain a single operator $T$ on a separable Hilbert space $\mathcal{H}$
with the property we need.
It is helpful to build $T$ so that there exists a unit vector $e \in \mathcal{H}$
so that 
\begin{equation} \label{evector}
T_je = e \text{ for all } j=1,\dots, d.
\end{equation}  
Namely, the first $d$-tuple
in our direct sum construction is the all ones $d$-tuple.
This gives $\langle T^{\alpha} e, e\rangle = 1$ which is 
convenient later.  

Now, the map $\C[z_1,\dots, z_d] \to B(\mathcal{H})$
given by $p\mapsto p(T)$ is a unital homomorphism
and we can view the range as equipped with a family of
matrix norms that allow us to identify $\C[z_1,\dots, z_d]$ 
with a subspace of $\mathcal{A}$.  
The associated norm is often called the Agler norm:
\[
\|P\|_{\mathcal{A}_d} := \|P(T)\|
\]
defined for matrix polynomials $P(z) \in \C^{M_1\times M_2}[z_1,\dots, z_d]$.
Again we emphasize that $\|P(T)\|$ is equal to the supremum of $\|P(S)\|$
where $S$ varies over all $d$-tuples of commuting contractive operators
or simply matrices.
 
 \begin{proof}[Proof of the `only if' direction of Theorem \ref{matrixthm}]
  Let $\mathcal{P}_L$ denote the set of polynomials
 of the form $P(z) = \sum_{\alpha \in L} P_\alpha z^{\alpha}$
 where necessarily $P_{\alpha} = 0$ for all but
 finitely many $\alpha$.
 Suppose $F(z)$ preserves matrix Agler classes $\mathcal{MA}_d(L)$
 under the $\star$ product.  
  Then, the map $\phi: \mathcal{P}_L\to \mathcal{P}_L$, 
  $\phi(P) = P\star F$ is completely bounded
 under the Agler norm in the sense that for any
 matrix polynomial $P(z) = ( p_{j,k}(z))_{j,k}$ for $p_{j,k} \in \mathcal{P}_L$
 \[
 \| (p_{j,k}\star F)_{j,k}\|_{\mathcal{A}_d} = \| P\star F\|_{\mathcal{A}_d} \leq
 \|P\|_{\mathcal{A}_d}.
 \]
Thus, the completely bounded norm $\|\phi\|_{cb}$ is at most $1$.
 It is sufficient to consider the case $\|\phi\|_{cb}=1$ for
 otherwise we can replace $F$ with $F/\|\phi\|_{cb}$
 to obtain our desired conclusion and then rescale the unit
 vectors in our desired conclusion by $\sqrt{\|\phi\|_{cb}}$
 to see that a general $F$ satisfies the correct formula.

 We may view $\mathcal{P}_L$ as sitting inside of $\mathcal{A}$ inside of $B(\mathcal{H})$.
 By the Wittstock extension theorem (Theorem \ref{wittstock}),
 $\phi$ extends to a completely bounded 
 map $\psi: \mathcal{A} \to B(\mathcal{H})$ with $\|\phi\|_{cb} = 1 = \|\psi\|_{cb}$.
 By the Stinespring representation theorem (Theorem \ref{stinespring})
 there exists a Hilbert space $\mathcal{K}$, isometries $V_j: \mathcal{H}\to \mathcal{K}$
 for $j=1,2$, and a $*$-homomorphism $\pi: \mathcal{A} \to B(\mathcal{K})$
 such that
 \[
 (p\star F)(T) = V_1^* \pi(p) V_2.
 \]
 Set $S_j = \pi(z_j)$.  The $d$-tuple $S = (S_1,\dots, S_d)$
 is a commuting contractive tuple since $\pi$ is a $*$-homomorphism
 and hence necessarily completely contractive.
 Then, for $\alpha \in L$,
 \[
 F_{\alpha} T^{\alpha} = V_1^* S^{\alpha} V_2
 \]
 and by \eqref{evector} we have
 \[
F_\alpha =  F_{\alpha} \langle T^{\alpha}e, e\rangle
  = \langle S^{\alpha} V_2 e, V_1 e\rangle
\]
and we can choose $x = V_2e, y= V_1e$ 
and
the theorem is finished. 
 \end{proof}
 
 \section{Acknowledgements}
 Professor Rien Kaashoek was a kind and caring person with an 
 infectious enthusiasm for life and its subfields: mathematics and operator theory.
 I will always remember that when initially asking me if I would help organize IWOTA 2016,
 he told me that I did not need to give a positive answer at that time, only a positive semi-definite one!
  This paper is dedicated to his memory. 
 
 Thanks also to: Jerry Li for many discussions on the topic of this paper
 and the anonymous referee for a thoughtful report.
 
 \subsection*{Data availability statement}
 There is no data associated with this article.

\subsection*{Funding and/or Conflicts of interests/Competing interests}
Partially supported by NSF grant DMS-2247702. Employed by Washington University in St. Louis. 
The author has no further relevant financial or non-financial interests to disclose.

\end{document}